\documentclass[]{article}

\usepackage[utf8]{inputenc}
\usepackage{amsmath, amssymb, amsthm} 
\usepackage{graphicx}                 
\usepackage{hyperref}                 
\usepackage{geometry}                 
\usepackage{tikz-cd}
\usepackage{enumitem}
\usepackage[style=numeric, sorting=none]{biblatex}
\usepackage{float}
\usepackage{subcaption}
\newtheorem{theorem}{Theorem}
\newtheorem{corollary}{Corollary}
\newtheorem{proposition}{Proposition}
\newtheorem{lemma}{Lemma}
\theoremstyle{definition}
\newtheorem{remark}{Remark}
\newtheorem{example}{Example}
\newtheorem{definition}{Definition}
\DeclareMathOperator{\ima}{Im}
\DeclareMathOperator{\lcm}{lcm}

\title{\textbf{Functorial Analysis of Group-Alphabet Cellular Automata}}

\author{
  Nathan Dion Larivière\thanks{Email: ndion23@ubishops.ca} $^{,1}$, 
  François Huard\thanks{Email: fhuard@ubishops.ca} $^{,1}$ \\[1ex]
  \small $^1$ Bishop's University, 2600 College street, J1M 1Z7, Québec, Canada.
}

\date{\small October 29, 2026}

\begin{document}

\maketitle

\begin{abstract}
    Let $G$ be a group and let $n\in\mathbb N$. We consider one-dimensional cellular automata of length $n$ over the alphabet $G$, with periodic boundary conditions and with local rules $\phi$ given by words. We construct a functor $F$ from $\mathbf{Grp}$, the category of groups (or any similar algebraic structure), to the category of pointed digraphs, and show its faithfulness. We introduce the image category $\mathcal C_{n,\phi}$ of state transition graphs of all such group-alphabet cellular automata, and show $F$ is an equivalence onto this category. Consequently we investigate the translation of categorical structures in $\mathbf{Grp}$ in terms of automata, and give examples of how standard methods of homological algebra may now be applied to automata. Additionally, a graph-theoretical construction of the product in $\mathcal C_{n,\phi}$ is given, together with a formula for cycle lengths and multiplicities in this product. When $G$ is abelian, further structural properties of state transition graphs follow, and we contrast these results with non-abelian counterexamples.
    
    \vspace{1ex}
    
\end{abstract}

\section{Introduction}
    Cellular automata (CA) were first invented by John von Neumann and Stanislaw Ulam in the 1940s as formal models for self-replication. 30 years later, John Conway introduced his well-known \textit{Game of Life}, which was popularized by Martin Gardner's column in \textit{Scientific American} and showed how complex behavior emerges from simple rules. In the 1980s Stephen Wolfram published a classification of 1D cellular automata, and introduced a naming scheme for "elementary" CA rules. Since then, cellular automata have developed into a broad area of mathematical research, with applications ranging from physical and biological modelling to computation and cryptography.
    
    The simplest example of a cellular automaton is Wolfram's Rule 90 automaton, where we consider a row of $n$ cells, each containing a value from the binary alphabet $\mathbb Z_2$. which evolves in discrete time steps according to the local rule:
    \[
    \phi(x)_i=x_{i-1}+x_{i+1}
    \]
    where $x=(x_1,\dots, x_n)$ is a configuration and $+$ denotes the group operation. In other words, to obtain the value in a cell at a certain time-step, perform the sum modulo 2 of the values in the left and right neighbor cells at the previous time, considering the $n$-th cell as left neighbor of the 1st, and vice-versa. Figure \ref{fig:R90_evolution} shows the successive evolution of a chosen initial configuration of this automaton, with $n=7$.
    \begin{figure}[htbp]
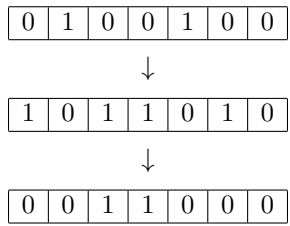

      \centering
      \[
      \begin{array}{c}
        \begin{array}{|c|c|c|c|c|c|c|}
          \hline 0 & 1 & 0 & 0 & 1 & 0 & 0 \\ \hline
        \end{array} \\[0.5em]
        \downarrow \\[0.5em]
        \begin{array}{|c|c|c|c|c|c|c|}
          \hline 1 & 0 & 1 & 1 & 0 & 1 & 0 \\ \hline
        \end{array} \\[0.5em]
        \downarrow \\[0.5em]
        \begin{array}{|c|c|c|c|c|c|c|}
          \hline 0 & 0 & 1 & 1 & 0 & 0 & 0 \\ \hline
        \end{array}
      \end{array}
      \]
      \caption{Two time step evolution of a Rule 90 automaton over $\mathbb Z_2$}
      \label{fig:R90_evolution}
    \end{figure}
    Since the state space is finite (in this case, it has size $2^n$), iterating successively must eventually reach some cycle. Despite the simplicity of the setup, predicting the behavior of such an automaton, for example the cycle lengths which occur, the number of these cycles, and the number of configurations which cannot be reached by evolution (\textit{Garden-of-Eden} (GOE) states), is far from being trivial. Algebraic analysis is performed in the seminal paper \cite{wolfram1984}.\newline
    CA are now studied in much more general forms, notably in \cite{cecchSMM} where the row of length $n$ is replaced with a group $G$ called \textit{lattice}, the binary values are replaced with a set called \textit{alphabet}, and the automaton evolves according to any of a broad class of \textit{local rules}. \newline
    In this paper, we restrict our attention to a class of CA which we call \textit{group-alphabet cellular automata} (GACA), where the alphabet is any group $G$, the lattice is $\mathbb Z_n$ (a row of $n$ cells with periodic boundary conditions), and the evolution map $\phi$ is given by a local rule represented by a formal word. An automaton can be associated with its \textit{state transition graph}, a functional digraph where nodes are configurations, and a directed edge $(x,y)$ exists whenever $x$ evolves to $y$ after one time step. \newline
    Whereas a similar class of automata called \textit{group cellular automata} (GCA) have been studied before, they are different from our class in that they require the evolution map to be a group homomorphism $\phi:G^n \to G^n$. In those cases, linear or algebraic methods can therefore be applied \cite{bardzellevohomomo}\cite{dennunziodyn}. For non-abelian groups, this structure is lost for almost all rules, and considerably less is known about the resulting state transition graphs. \newline
    A key contribution of this paper is the construction, for a fixed $n, \phi$, of a functor  
    \[
    F_{n,\phi}:\mathbf{Grp} \to \mathbf{pfG}\]
    which associates a group $G$ to the state transition graph $F(G)$ of the corresponding automaton. Here $\mathbf{Grp}$ is the category of groups and group homomorphisms and $ \mathbf{pfG}$ is the category of pointed functional digraphs and pointed functional digraph homomorphisms. When restricting the target category to the subcategory $\mathcal{C}_{n, \phi}$ of state transition graphs of $(n,\phi)$-GACA with pointed digraph morphisms of the form $F(f)$ for $f:G \to K$ a group homomorphism, $F$ is an equivalence of categories. This immediately gives rich structure to the category $\mathcal{C}_{n, \phi}$, and we give its important categorical constructions. Notably, we show the graph-theoretic construction of the categorical product in $\mathcal{C}_{n, \phi}$ is the graph tensor product, that is:
    \[
    F(G\times H)\cong F(G)\otimes F(H),
    \]
    where $\otimes$ is the tensor, or categorical, product of graphs. This turns algebraic decompositions of groups into graph decompositions. Since the cyclic structure of tensor products of functional graphs can be described explicitly, this gives formulas for the cycle lengths and multiplicities of $F(G\times H)$ in terms of those of $F(G)$ and $F(H)$.\newline
    Furthermore, we show examples of how one may, using this functorial framework, apply standard techniques of homological algebra to automata, something which to our knowledge has not been done before. For example, we use the fact that the functor preserves short exact sequences to build commutative diagrams of automata and get structural information about semi-direct products, for instance.

\section{Group-Alphabet Cellular Automata}
    Throughout this paper, we let $G$ be a group (not necessarily abelian), and $*$ denote the group operation. We let $e$ be the identity element in $G$. The lattice is a row of length $n$, so that a configuration is an element $x=(x_1,\dots,x_n) \in G^n$. We call \textit{identity node} the node $e^n=(e,\dots,e)$ with the group element in every position. We set periodic boundary conditions, i.e. the lattice may be identified with $\mathbb Z_n$. One may view these automata as ring-shaped configurations evolving down on a cylinder. The global evolution map is given by a local rule defined as follows:
    Throughout this paper, we fix a neighbourhood
\[
I=(j_1,\dots,j_k)\in\mathbb Z_n^k
\]
and a word
\[
w(X_1,\dots,X_k),
\]
which may involve any integer powers of the abstract variables $\{X_1, \dots , X_k\}$.

\bigskip

For example $X_3^2X_1^{-2}X_3X_4$ is a  word in the variables  $\{X_1, X_2, X_3 , X_4\}.$ 
 \begin{definition}[Global evolution map]
 For a group $(G,*)$ define the global evolution map
\[
\phi_w:G^n\to G^n
\]
by
\[
\phi_w(x)_i
=
w(x_{i+j_1},\dots,x_{i+j_k}),
\]
with indices taken modulo $n$, and where the concatenation $x_i^mx_j^n$ is the element $x_i^m*x_j^n$ of $G$.
\end{definition}
    Note that we consider $w$ to be a formal word, in the sense that a rule given by a word $w$ may be applied to any group. For example, the word in the following example may be used in defining a global map for an automaton over any group, when the $x_i$ are taken from it and $*$ considered its group operation. Note also that this class of rules respects the quiescence condition, that is the node $e^n$ always evolves to itself after one time step.
    \begin{definition}
    In this paper, a cellular automaton is considered a triple $(G,n,\phi)$, where $G$ is the alphabet group, $n \in \mathbb N$ the lattice length, and $\phi$ the global evolution map (induced by the word $w$ and $I$).
    \end{definition}
    \begin{example}
    Wolfram's Rule 90 automaton can be defined as follows:
    \[
   (\mathbb Z_2, n, \phi_{90}) \text{ where} \quad n\in \mathbb N, \quad I = (-1,1), \quad w(X_1,X_2)=X_1X_2, 
    \]

   \text{therefore}\quad
   \[
    \phi_{90}(x)_i=w(x_{i-1},x_{i+1})=x_{i-1}+ x_{i+1}.
    \]
    
    Detailed algebraic analysis of this automaton and more is given in the previously cited \cite{wolfram1984}.

     An immediate observation is that if $(G,\cdot)$ is any group and $f:\mathbb Z_2 \to G$ is a group homomorphism, then $f$ commutes with $w$:
     
         $$f(w(x_{i-1},x_i))=f(x_{i-1}+x_{i+1})=f(x_{i-1})\cdot f(x_{i+1})=w(f(x_{i-1}),f(x_{i+1})).$$

\end{example}
\begin{example} 
Consider now the (non abelian) dihedral group $(D_5, \cdot)$, and the following automaton:
 \[
   (D_5, n, \phi) \text{ where} \quad n\in \mathbb N, \quad I = (2,0,-1), \quad w(X_1,X_2,X_3)=X_3X_2X_1^{-3}X_2, 
    \]
    \[
    \text{therefore}\quad
    \phi(x)_i=w(x_{i+2},x_i,x_{i-1})=x_{i-1}\cdot x_i\cdot x_{i+2}^{-3}\cdot x_{i}.
    \]
Note that since $D_5$ is not abelian, distinct orderings of $I$ and $w$ will yield  different automata.  

Again, it is straightforward to verify that if $G$ is any group and $f:D_5\to G$ is a group homomorphism, then $f$ commutes with $w$.  This property clearly holds  regardless of the groups and  evolution map involved, and will play a crucial role in what follows.
    
\end{example}
An important representation of a cellular automaton, and the main object of study of this paper is the state transition graph, which gives a global view of the behavior of an automaton.

\begin{definition} [Pointed functional digraph]
        A \textit{functional digraph} is an oriented graph where every node has an outdegree of exactly one. Equivalently, it is a set $V$ with a function $V \to V$, hence the name functional digraph. A \textit{pointed functional digraph} is a functional digraph with a distinguished vertex, called basepoint. A morphism of pointed functional digraphs is an adjacency-preserving and basepoint-preserving vertex map. 
        Pointed functional graphs and their morphisms form a category, which we denote $\mathbf{pfG}$.
    \end{definition}
    We can now define the main object of this paper:
    \begin{definition} [State transition graph]
        The \textit{state transition graph} (or state transition diagram) of an automaton $(G, n, \phi)$ is the directed graph
        \[ F_{n,\phi}(G)\]
        with vertex set
        \[ V(F_{n,\phi}(G))=G^n\]
        and edge set
        \[ E(F_{n,\phi}(G))=\{(x,\phi(x)): x \in G^n\}\]
    \end{definition}
    Every node has a single outgoing edge, hence $F_{n,\phi}(G)$ is a functional digraph. Moreover, the identity node $e^n$ is distinguished, hence we consider $F_{n,\phi}(G)$ as a pointed functional digraphs with basepoint $e^n$.
    Configurations with no predecessors are called Garden of Eden (GOE) configurations. When $G$ is abelian, the state transition graphs exhibit many nice symmetries, for example the evolution map $\phi:G^n \to G^n$ is a group homomorphism, hence the indegree of all non-GOE nodes is exactly  $\lvert \ker \phi \rvert$ \cite{bardzellevohomomo}. In general however, that does not hold, hence other methods of analysis are required.
    
    \section{Functor construction}
    We now introduce a map $F$ associating a group $G$ to its state transition graph $F_{n,\phi}(G)$ for a given $n, \phi$. When a map on morphisms is added, we show it acts as a functor from the category $\mathbf{Grp}$ of groups to the category $\mathbf{pfG}$ of pointed functional graphs.\newline
    
    We define the mapping \[
    F_{n,\phi}:\mathbf{Grp} \to \mathbf{pfG}\]
    by:
    \begin{enumerate}
        \item To each group $G$, we assign the digraph $F_{n,\phi}(G)$.
        \item To each group homomorphism $f: G \to H$, we assign the map $F_{n,\phi}(f)$ defined by:
        \[ \forall x=(x_1,...,x_n)\in G^n, \quad F_{n,\phi}(f)(x_1, \dots, x_n) = (f(x_1), \dots, f(x_n))\in H^n \]
    \end{enumerate}
    Unless confusion is at risk, we will henceforth simply use $F(G)$ and $F(f)$ in place of $F_{n,\phi}(G)$ and $F_{n,\phi}(f)$. We now verify that $F$ satisfies the functorial axioms:
    \begin{proof}[Verification]\leavevmode
    \begin{enumerate}
    
        \item \textbf{Morphism Validity:} We must check that $F(f)$ is a pointed graph homomorphism. Let $(x,y)$ be an edge in $F(G)$, i.e. $\phi(x) = y$. Because $f$ is a group homomorphism, it commutes with $w$, and for each $i$ we have:
        \[ \phi(F(f)(x))_i = \phi(f(x_1),\dots,f(x_n))_i = w(f(x_{i+j_1}),\dots,f(x_{i+j_k})) = \]
        \[f\left( w(x_{i+j_1},\dots,x_{i+j_k})\right) = f(\phi(x)_i)=F(f)(\phi(x))_i  =F(f)(y)_i \]
        Hence, $\phi(F(f)(x)) = F(f)(y)$, meaning $\big(F(f)(x), F(f)(y)\big)$ is an edge in $F(H)$, so $F(f)$ is a graph homomorphism. Furthermore, $F(f)(e^n)=(f(e))^n=(e')^n$ for $e'$ the identity in $H$, which makes $F(f)$ a pointed graph homomorphism.

        \item \textbf{Identity Preservation:} For $1_G:G\to G, \quad 1_{F(G)}:F(G)\to F(G)$
        \[ F(1_{G})(x) = (1_{G}(x_1), \dots, 1_{G}(x_n)) = (x_1, \dots, x_n) = 1_{F(G)}(x) \]
        
        \item \textbf{Composition Preservation:} For $f:G \to H, \quad g: H \to L, \quad x \in G^n$
        \[ F(g \circ f)(x) = (g \circ f(x_1), \dots, g \circ f(x_n)) =F(g)(f(x_1),\dots, f(x_n))\]
        \[= F(g)\big(F(f)(x)\big) = (F(g) \circ F(f))(x) \]
        
    \end{enumerate}
    \end{proof}
     From (1), we get that $\phi(F(f)(x))=F(f)(\phi(x))$ whenever $f$ is a group morphism.  Note that this functor may clearly be generalized to be acting on any similar algebraic structure, such as rings, algebras, fields, modules. The crucial point is that the morphisms of the structure must commute with the evolution rule, which is obvious in these cases. We now give some basic properties of this functor.
    
    \begin{proposition}
        Let $f:G\to H$ be a group morphism.  Then $F(f)$ is injective (surjective) if and only if $f$ is injective (surjective).
        \begin{proof}
            Assume that $f$ is not injective.  Then there exist $x\neq y\in G$ such that $f(x)=f(y)$.  Hence $F(f)(x^n)= F(f)(y^n)$ and $F(f)$ is not injective.  Now if $f$ is injective, let $x=(x_1, \dots, x_n)$ and $y=(y_1,\dots, y_n)$ be in $G^n$, such that $F(f)(x)=F(f)(y)$, then for all $i$, we have $f(x_i)=g(y_i)$ hence $x_i=y_i$, so $x=y$, and $F(f)$ is injective.
    \smallskip
    
    Next assume that $f$ is surjective and let $(y_1,\dots ,y_n)$ be in $F(H)$.  Then for each $i$ there exists $x_i\in G$ such that $f(x_i)=y_i$.  This implies that $F(f)(x_1,\dots ,x_n)=(y_1, \dots , y_n)$, thus $F(f)$ is surjective.  Conversely, if $F(f)$ is surjective, let $y\in H$. Then $(y,\dots , y) \in F(H)$ and there exists $(x_1,\dots, x_n)\in F(G)$ such that $F(f)(x_1,\dots, x_n)=(y,\dots, y)$.  In particular, $f(x_1)=y$, and $f$ is surjective.

        \end{proof}
    \end{proposition}
    
    Recall that functor between two categories $C$ and $D$ is faithful if for any two maps $f,g:X\to Y$ in $C$, $F(f)=F(g)$ implies $f=g$.
     
    \begin{proposition}
        $F$ is a faithful functor.
    \end{proposition}
    \begin{proof}
        Assume $F(f)=F(g)$ for $f,g:G\to H$ two group homomorphisms. Then $F(f)(x)=F(g)(x)$ for all $x \in G^n$, or $(f(x_1),\dots,f(x_n))=(g(x_1),\dots,g(x_n))$, hence in particular $f(x_1)=g(x_1)$ for all $x_1 \in G$, hence $f=g$.
    \end{proof}
    \begin{remark}
        When considering $\mathbf{pfG}$ as the category of all pointed functional digraphs with all pointed graph homomorphisms, $F$ is not full. Indeed, take for example the graphs $F_{3,\phi_{60}}(S_3)$, and $ F_{3,\phi_{60}}(\mathbb Z_3)$ shown below, with $\phi_{60}(x)_i=x_{i-1}*x_i$. It is not hard to construct between them a non-trivial pointed graph homomorphism, yet there exist no non-trivial group homomorphisms $f:S_3 \to \mathbb Z_3$.
    \begin{figure}[H]
        \centering
        \begin{subfigure}[c]{0.48\textwidth}
            \centering
            \includegraphics[width=\linewidth]{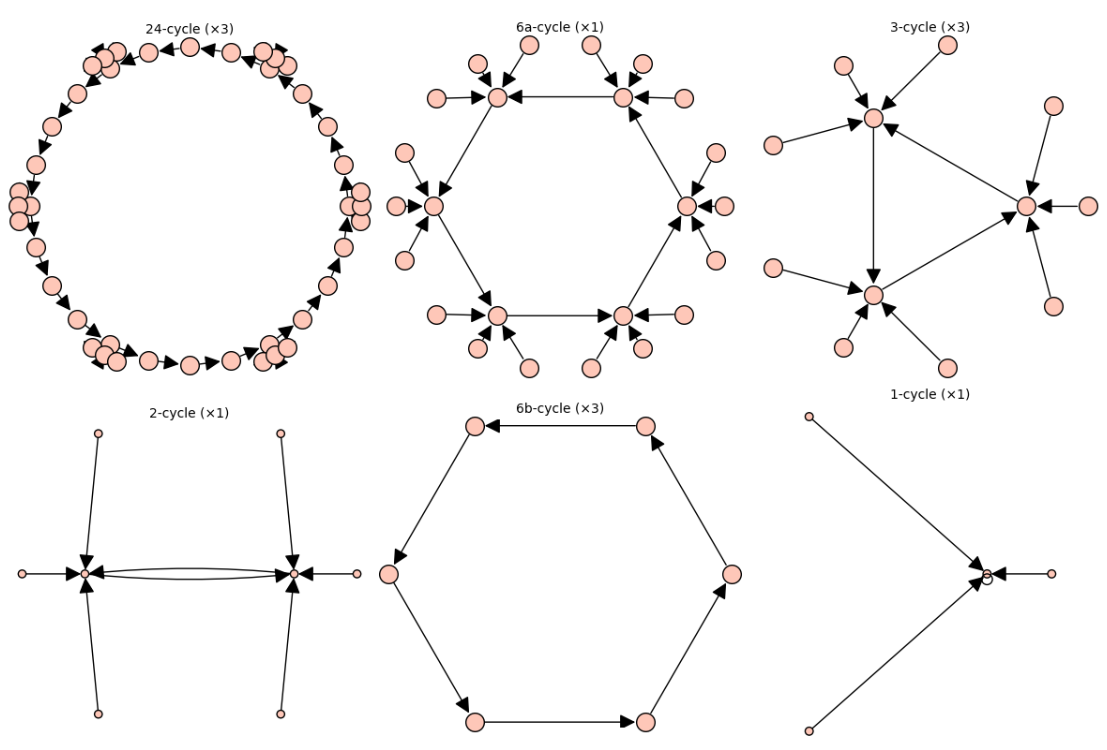}
            \label{fig:placeholder1_1}
        \end{subfigure}
        \hfill 
        \begin{subfigure}[c]{0.48\textwidth}
            \centering
            \includegraphics[width=\linewidth]{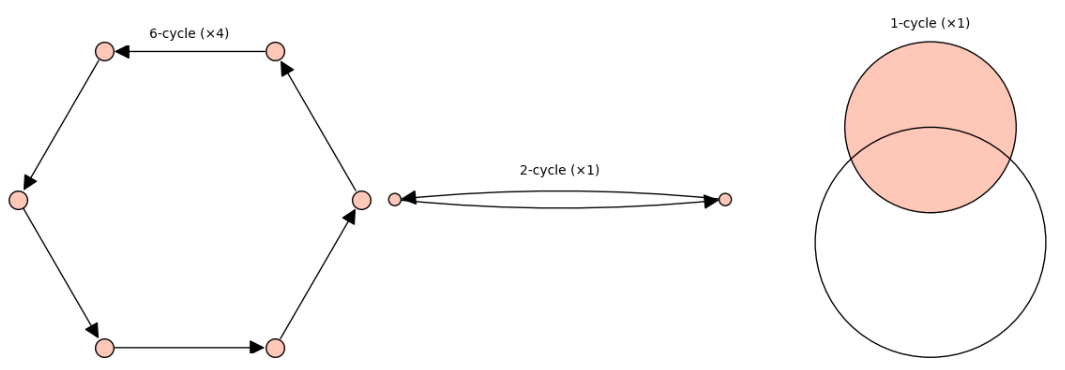}
            \label{fig:placeholder2_1}
        \end{subfigure}
        \caption{The state transition graphs $F_{3,\phi_{60}}(S_3)$ and $F_{3,\phi_{60}}(\mathbb Z_3)$} 
        \label{fig:main_label_1}
    \end{figure}

    \end{remark}

    \begin{remark}
        The strict image of a functor is in general not a subcategory of the target. For instance, let $\mathcal C$ have objects $A,B,C,D$ and non-identity morphisms $A\to B$ and $C\to D$, let $\mathcal D$ have objects $X,Y,Z$ and non-identity morphisms $X\to Y$, $Y\to Z$ and their composite $X\to Z$, and let $F$ send $A,B,C,D$ to $X,Y,Y,Z$. The image contains $X\to Y$ and $Y\to Z$ but not $X\to Z$. The obstruction is that $F$ is not injective on objects. Note that this is different from the so-called essential image or the full subcategory generated by the image, both of which are always subcategories of the target category.
    
        In our setting we adopt the harmless set-theoretic convention that two distinct groups have distinct underlying sets (any group may be replaced by an isomorphic copy, e.g.\ $G\times\{G\}$, to ensure this). Then $G\neq H$ implies $G^n\neq H^n$, hence $F(G)\neq F(H)$, so $F$ is injective on objects. If $F(f_1)$ and $F(f_2)$ are composable in $\mathbf{pfG}$, then $F(\operatorname{cod} f_2)=F(\operatorname{dom} f_1)$, so $\operatorname{cod} f_2=\operatorname{dom} f_1$ and $F(f_1)\circ F(f_2)=F(f_1\circ f_2)$ lies in the image. Thus the image of $F$ is indeed a subcategory of $\mathbf{pfG}$.

        \bigskip
        
    In view of this, if we restrict ourselves to the strict image $\mathbf{F(Grp)}$ of the functor $F$, that is only consider pointed functional digraphs of the form $F(G)$ for some group $G$ with morphisms of the form $F(f)$ for $f$ a group homomorphism, then $F$ becomes full and  surjective on objects (dense), in addition to being faithful as previously proved. Thus $F$ is an equivalence between the categories $\mathbf{F(Grp)}$, which we denote by $\mathbf{\mathcal{C}_{n, \phi}}$, and $\mathbf{Grp}$. Consequently, any purely categorical construction (kernels, cokernels, products, coproducts, pullbacks, etc) that exists in one category automatically exists in the other. The functor also preserves short exact sequences, which enables us, as we will see later,  to use tools from homological algebra to transfer information from groups to automata. For example, a consequence of this is that the category $F(\mathbf{Ab})$ of automata over abelian groups is itself abelian, which are amongst the most understood categories.

    One might consider that the requirement that morphisms between automata be of the form $F(f)$ is a heavy restriction, but in general, the interest in automata lies in the objects rather than the morphisms. We now turn our attention to the construction of the important objects in this category.
\end{remark}
\section{Categorical aspects}
In what follows, the target of the functor $F$ is the subcategory $\mathbf{\mathcal{C}_{n, \phi}}$ of {\bf pfG} formed by the image of the functor $F$.  Since $F$ is an equivalence of categories, we have that all categorical structures in $\mathbf{Grp}$ are preserved in $\mathbf{\mathcal{C}_{n, \phi}}$.

    \subsection{Epis and monos}

    Recall that a morphism $f:G\to H$ is a {\bf monomorphism} if for any pair of morphisms $g_1, g_2:K\to G$, the condition $f\circ g_1= f\circ g_2$ implies $g_1=g_2$.  The notion of epimorphism is obtained by reversing all the arrows.  In the category of groups a  morphism is an epimorphism (monomorphism) if and only if it is  surjective (injective). This is not true in general. For example, in the category of Rings, an epimorphism is not necessarily surjective. Also, note that an equivalence of categories does not necessarily preserve injectivity and surjectivity.   This is true in our case and it implies the following result.
    
    \begin{corollary}
         $F$ preserves epimorphisms and monomorphisms.  Moreover, $F(f)$ is an epimorphism (monomorphism) if and only if $F(f)$ is surjective (injective).
    \end{corollary} 

        \begin{proof}
            The first part follows from the fact that $F$ is an equivalence of categories.  The second part then follows from Proposition 1. 

        \end{proof}
        
\subsection{Zero object}
    In a category, a {\bf zero object} is an object that is both initial and terminal, that is an object from which a unique morphism exists to any object, and to which a unique morphism exists from any object. In the category of groups, there is a zero object, namely the trivial group $\{e\}$. It follows that in $\mathbf{\mathcal{C}_{n, \phi}}$, the image of the trivial group, $F(\{e\})$, is also a zero object. This is a graph consisting of the single node $(e,\dots, e)$ with a self loop. 
    
    Indeed, one can easily see that the initial morphism can be constructed from $e^n$ to any automaton $F(G)$ by mapping basepoint to basepoint.  Note that this morphism is unique in the target category.  Indeed, if the graph $e^n$ maps to another node, then this node must also be a loop, and if this loop is not $e^n$, then the map cannot be of the form $F(f)$ for $f:\{e\} \to G$ a group morphism since $f(e)=e_G$.  For similar reasons, the morphism from any automaton $F(G)$ to $e^n$ that maps every node to the basepoint is unique in the target category.  These initial and terminal morphisms are often referred to as zero morphisms, since they both map all the elements to the identity.

    Note that this differs from the broader category of graphs in which the initial object is the empty graph, and no terminal object exists.

    \subsection{Kernels}
    Recall that if $f:G\to H$ is a  morphism, then a morphism $k:K\to G$ is a {\bf kernel} of $f$ if 
    \begin{itemize}
    \item[(i)] $f\circ k$ is the zero morphism, and
    \item[(ii)] Given any morphism $k':K'\to G$ such that $f\circ k'$ is the zero morphism, there is a unique morphism $u:K'\to K$ such that $k\circ u=k'$. 
    \end{itemize}

    In the category of groups, the categorical kernel of $f$ coincides with the standard algebraic kernel of $f$, that is $K=\ker f=\{x\in G : f(x)=e_H\}$ together with the inclusion map $i:\ker f \to G$.  Since $F$ is an equivalence of categories, it maps kernels to kernels, hence we immediately get that $\mathbf{\mathcal{C}_{n, \phi}}$ has kernels and that 
    
    \begin{corollary}
        The kernel of the map $F(f):F(G)\to F(H)$ is given by $F(i): F(\ker f)\to F(G).$
    \end{corollary}
          It follows from above and from the definition of the functor $F$ that $F(\ker f)$ is the graph with vertex set $(\ker f)^n$ and edges given by the evolution map $\phi$. Note that this agrees with the usual notion of kernel of a pointed graph morphism, which is defined to be the preimage of the basepoint $e^n$. Indeed, we have $\{x\in G^n : F(f)(x)=e_H^n\}=\{(x_1,...,x_n)\in G^n : f(x_i)=e_H \text{ for all i}\} = \{(x_1,...,x_n):x_i\in \ker f \text{ for all i}\}=(\ker f)^n$.\newline

         \smallskip
         
         As a sample, we provide the proof that $F(i): F(\ker f)\to F(G)$ satisfies the universal property of a Kernel in $F(\mathbf{Grp})$.  This is not necessary since we have an equivalence of categories, but it shows the typical gymnastic of basic category theory:
         \begin{proof}
             Take a morphism 
                \begin{tikzcd}
                F(G) \arrow[r, "F(f)"] & F(H)
                \end{tikzcd} 
            in $\mathbf{\mathcal{C}_{n, \phi}}$, then from the inclusion map $i:\ker f \to G$,  we have the morphism $F(i):F(\ker f) \to F(G)$. Now for all $x \in F(\ker f), \,\, x_j \in \ker f$ and so $F(f)\circ F(i)(x)=F(f)(x)=(f(x_1), \dots, f(x_n))=(e,\dots, e) \implies F(f)\circ F(i)=0$. Now assume there exists another object $F(K')$ and a morphism $F(k'):F(K') \to F(G)$ such that $F(f)\circ F(k')=0$. Then $F(f\circ k')=0=F(0)$ so by faithfulness of $F$,  $ f\circ k'=0$, and by the universal property of $\ker f$ in $\mathbf{Grp}$, there exists a unique morphism $u:K' \to \ker f$ such that $i \circ u=k'$. It follows (again by faithfulness) that there exists a unique morphism $F(u):F(K') \to F(\ker f)$ such that $F(i)\circ F(u)=F(k')$.
        \end{proof}
                \begin{center}
                \begin{tikzcd}
                    \ker f \arrow[r, "i"]                                & G \arrow[r, "f"] & H \\
                    K' \arrow[ru, "k'"'] \arrow[u, "\exists! u", dashed] &                  &  
                    \end{tikzcd} 
                        \begin{tikzcd}
                        {} \arrow[r, "F", Rightarrow] & {}
                        \end{tikzcd}
                    \begin{tikzcd}
                    F(\ker f) \arrow[r, "F(i)"]                                   & F(G) \arrow[r, "F(f)"] & F(H) \\
                    F(K') \arrow[ru, "F(k')"'] \arrow[u, "\exists! F(u)", dashed] &                        &     
                    \end{tikzcd}
                \end{center}

    \begin{example}
        Consider the parity morphism $\delta:S_3 \to \mathbb Z_2$, whose kernel is $\ker \delta = A_3$, and its image in $\mathcal C_{3, \phi{60}}$ $F(\delta):F(S_3) \to F(\mathbb Z_2)$ with $\ker F(\delta)=F(A_3)$ as shown below.
     
\begin{figure}[H]
    \centering
    \begin{subfigure}[c]{0.48\textwidth}
        \centering
        \includegraphics[width=\linewidth]{fig_s3_r60_n3.png}
        \caption{The state transition graph $F_{3,\phi_{60}}(S_3)=\operatorname{dom}F(\delta)$}
        \label{fig:placeholder1_2}
    \end{subfigure}%
    \hfill
    \begin{subfigure}[c]{0.48\textwidth}
        \centering
        \includegraphics[width=\linewidth]{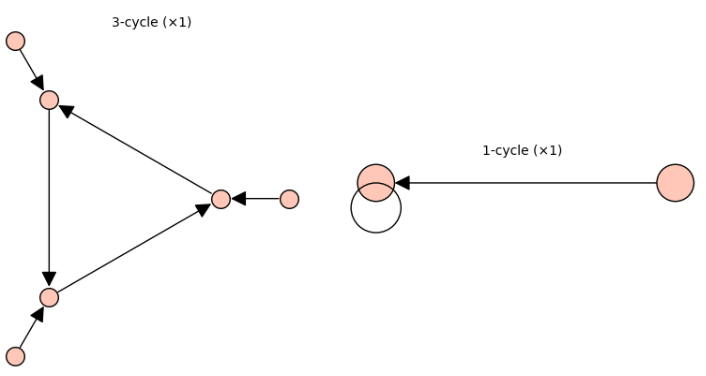}
        \caption{The state transition graph $F_{3,\phi_{60}}(\mathbb{Z}_2)=\operatorname{codom}F(\delta)$}
        \label{fig:placeholder2_2}
    \end{subfigure}

    \vspace{1em} 

    \begin{subfigure}[c]{0.6\textwidth}
        \centering
        \includegraphics[width=\linewidth]{fig_z3_r60_n3.png}
        \caption{The state transition graph $F_{3,\phi_{60}}(A_3)=\ker F(\delta)$}
        \label{fig:placeholder3_2}
    \end{subfigure}

    \caption{Example of a kernel graph.} 
    \label{fig:main_label_2}
\end{figure}
One can readily see the inclusion of the graph in (c) into the graph of (a). Furthermore, we automatically know that nodes in the cycles of $F_{3,\phi_{60}}(S_3)$ which do not appear in the kernel must contain odd permutations, else they would have been sent to the node $e_{\mathbb Z_2}^n$ and would appear in the kernel graph.
\end{example}
    \subsection{Cokernels}
The cokernel is the dual notion of the kernel.  If $f:G\to H$ is a morphism, then a morphism $q:H\to Q$ is a {\bf cokernel} of $f$ if 
    \begin{itemize}
    \item[(i)] $q\circ f$ is the zero morphism, and
    \item[(ii)] Given any morphism $q':H\to K'$ such that $q'\circ f$ is the zero morphism, there is a unique morphism $u:K\to K'$ such that $u\circ q=q'$. 
    \end{itemize}

     In the category of groups, the categorical cokernel of $f$ always exists and is the morphism $q:H\to H/N$  which maps $h$ to $hN$, where $N$ is the normal closure of $f(G)$, that is the smallest normal subgroup of $H$ containing $f(G)$. It then follows that $F(q):F(H)\to F(H/N)=(H/N)^n = \{(x_1N,\dots, x_nN) : x_i\in H\}$ is a cokernel in $\mathcal{C}_{n, \phi}$.  The object $F(H/N)$ is given by the following quotient graph:
    
        \begin{definition}[Quotient graph]
            Given $N$ a normal subgroup of a group $H$, we define the quotient graph $F(H)/\sim_N$ by identifying the vertices of $F(H)$ together according to the following equivalence relation:
            \[
            (x_1,...,x_n)\sim_N(y_1,...,y_n) \iff x_iN=y_iN \text{ for each } i \in \{1,...,n\}
            \]
            One can readily check that $\sim_N$ satisfies the axioms of an equivalence relation.
            
            We denote by $ [x_1,...,x_n] $ the equivalence class of $(x_1,\dots,x_n)$ under this relation.  The equivalence classes of $F(H)/\sim_N$ are the vertices of the quotient graph.  Moreover there is an edge from $ [x_1, \dots,x_n] $  to $[y_1,\dots , y_n]$ if $\phi(x_1,\dots, x_n)=(y_1,\dots, y_n)$.  It is then straightforward to check that this is well defined and that the graphs $F(H/N)$ and $F(H)/\sim_N$ are isomorphic.
            
        \end{definition}


        \begin{example}
           Consider now the inclusion morphism $f:\mathbb Z_4   \to D_4$ given by $n \mapsto \rho^n $ for each $n \in \mathbb Z_4$ (with $\rho$ the rotation generator in $D_4$), whose cokernel is $\operatorname{coker}(f)=D_4/\langle\rho\rangle \cong \mathbb Z_2$. ($\mathbb Z_4$ is isomorphic to $ \langle \rho \rangle$ which is normal in $ D_4$). We give the images in $\mathcal C_{5,\phi_{90}}$ of these three groups below.
           \begin{figure}[H]
    \centering
    \begin{subfigure}[c]{0.48\textwidth}
        \centering
        \includegraphics[width=\linewidth]{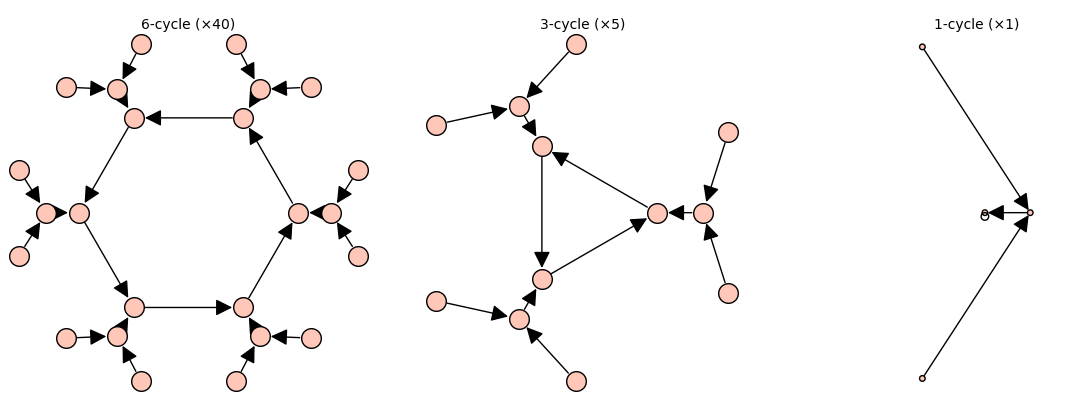}
        \caption{The state transition graph $F_{5,\phi_{90}}(\mathbb Z_4)=\operatorname{dom}F(f)$}
        \label{fig:placeholder1_3}
    \end{subfigure}%
    \hfill
    \begin{subfigure}[c]{0.48\textwidth}
        \centering
        \includegraphics[width=\linewidth]{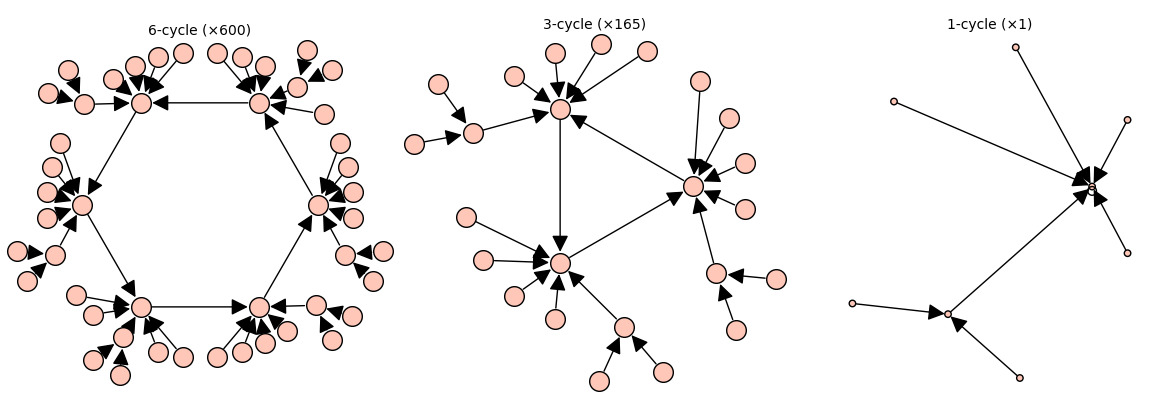}
        \caption{The state transition graph $F_{5,\phi_{90}}(D_4)=\operatorname{codom}F(f)$}
        \label{fig:placeholder2_3}
    \end{subfigure}

    \vspace{1em} 

    \begin{subfigure}[c]{0.6\textwidth}
        \centering
        \includegraphics[width=\linewidth]{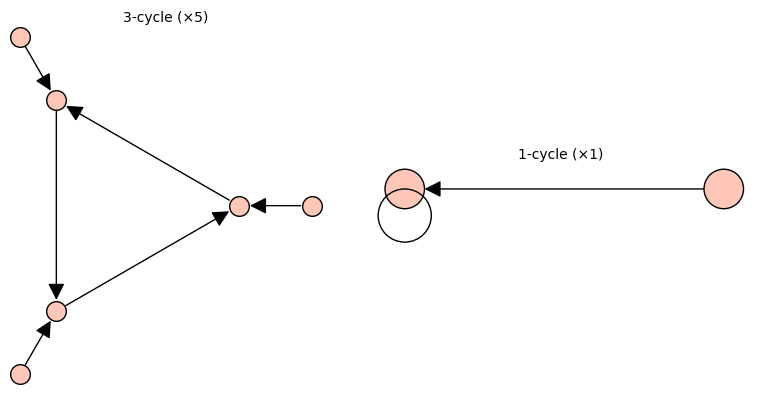}
        \caption{The state transition graph $F_{5,\phi_{90}}(\mathbb Z_2)=\operatorname{coker} F(f)$}
        \label{fig:placeholder3_3}
    \end{subfigure}

    \caption{Example of a cokernel graph.} 
    \label{fig:main_label_3}
\end{figure}

        \end{example}

    \subsection{Products}

The \textbf{categorical product} of a family of objects $\{G_i\}_{i \in I}$ is an object $P$ together with a family of  morphisms $\{\pi_i:P\to G_i\}_{i\in I}$ that satisfy the following universal property: given a family of morphisms $\{f_i: H \to G_i\}_{i\in I}$, there exists a unique morphism $f: H \to  P$ such that the following diagram commutes for all $i \in I$:

            \begin{center}
                \begin{tikzcd}
                    P\arrow[r, "\pi_i"]                                & G_i \\
                    H \arrow[ru, "f_i"'] \arrow[u, "\exists! f", dashed] &                  
                    \end{tikzcd} 
                   
                \end{center}

  The product in the category of groups exists and is the standard direct product $\prod_{i \in I} G_i$, together with the projections $\pi_i : \prod_{i \in I} G_i\to G_i$.  It consists of all Cartesian tuples or sequences $(g_i)_{i \in I}$ where $g_i \in G_i$ for each $i \in I$, under component-wise multiplication. 

    \begin{corollary}
        The product of a family $\{F(G_i)\}_{i\in I}$  in $\mathcal{C}_{n, \phi}$ exists and is $F(\prod_{i\in I} G_i)$, together with the morphisms $F(\pi_i) : F(\prod_{i \in I} G_i)\to F(G_i)$.
    \end{corollary}
    Products of cellular automata, as studied in \cite{castillocatprod}, \cite[pp. 122066]{dennunziodyn}, are relatively well known. In the broader category of graphs and graph morphisms, the categorical product is the tensor product, which we recall.
        \begin{definition}[\cite{hammackhandbook} Tensor product of graphs]
            The tensor product of two graphs $\Gamma_1$ and $\Gamma_2$ is the graph, denoted in this paper as $\Gamma_1 \otimes \Gamma_2$, with the following vertex and edge sets. 
            \[V(\Gamma_1 \otimes \Gamma_2)=V(\Gamma_1)\times V(\Gamma_2)\]
            \[E(\Gamma_1 \otimes \Gamma_2)=\Big\{\big((g,h),(g',h')\big):(g,g')\in E(\Gamma_1) \text{ and } (h,h') \in E(\Gamma_2)\Big\}\]
            (This product is also known as the Kronecker product, direct product, or categorical product).
        \end{definition}
    As it turns out, this operation is also the product in our category $\mathbf{\mathcal{C}_{n, \phi}}$.
        \begin{theorem}
            For any two groups $G_1$ and $G_2$, we have that $F(G_1 \prod G_2) \cong F(G_1) \otimes F(G_2)$. In other words, the categorical product in $\mathbf{\mathcal{C}_{n, \phi}}$ is the graph tensor product.
        \end{theorem}
        \begin{proof}
            We construct the canonical vertex isomorphism $\varphi_V:(G_1 \prod G_2)^n \to G_1^n\times G_2^n$ defined by \newline
            $\varphi\big((a_1,b_1),\dots,(a_n,b_n)\big)=\big((a_1,\dots,a_n),(b_1,\dots,b_n)\big)$. Now observe that 
            
            \[
            (u,v)=\Big[\big((a_1,b_1),\dots,(a_n,b_n)\big),\big((c_1,d_1),\dots,(c_n,d_n)\big)\Big]
            \]
            is an edge in $F(G_1 \prod G_2)$ if and only if $(c_i,d_i) = w\big((a_{i+j_1},b_{i+j_1}),\dots,(a_{i+j_k},b_{i+j_k})\big) =\newline
        (w(a_{i+j_1},\dots,a_{i+j_k}), w(b_{i+j_1},\dots,b_{i+j_k}))$ $\iff$ $c_i =w(a_{i+j_1},\dots,a_{i+j_k})$ and $d_i = w(b_{i+j_1},\dots,b_{i+j_k})$ $\iff$ $(a,c)$ is an edge in $F(G_1)$ and $(b,d)$ is an edge in $F(G_2)$ $\iff$ $\big((a,b),(c,d)\big)$ is an edge in $F(G_1)\otimes F(G_2)$.
        \end{proof}

        As a consequence of this result, we obtain the construction of the categorical product of any finite family of objects $F(G_i)$ in $\mathbf{\mathcal{C}_{n, \phi}}$.

        \begin{remark}
            It is a known fact \cite{roylealgebraicgraph} that the adjacency matrix   of the graph $\Gamma_1 \otimes \Gamma_2$ is given by the Kronecker product of the adjacency matrices of the factor graphs. That is:
            \[
            A(\Gamma_1 \otimes \Gamma_2)=A(\Gamma_1)\otimes_{Kronecker}A(\Gamma_2)
            \]
            Where $A(\Gamma)$ is the adjacency matrix of the graph $\Gamma$.
        \end{remark}

    \subsection{Coproduct}

    The categorical coproduct of a family of objects $\{G_i\}_{i\in I}$ is obtained by reversing all the arrows in the definition of the product, and is denoted by $\{q_j: G_j\to\oplus_{i\in I}G_i\}_{j\in I}$.  In the category of groups, the categorical coproduct of $\{G_i\}_{i \in I}$ exists and is their free product, that is the group whose elements are finite words $g_1g_2\dots g_m$ where each $g_k\in G_{i_k}$ and adjacent indices $i_k, i_{k+1}$ are distinct, and whose operation is the concatenation.
    \begin{corollary}
        The categorical coproduct of a family $\{F(G_i)\}_{i\in I}$  in $\mathcal{C}_{n, \phi}$ exists and is $F(\oplus_{i\in I} G_i)$, together with the morphisms $F(q_j) : F(G_j)\to F(\oplus_{i \in I}G_i$).
    \end{corollary}
        In the general category of pointed graphs, the coproduct (or sum) is normally defined as the wedge sum, which is the disjoint union with basepoints identified together. For two finite groups $G,K$, the wedge sum of $F(G)$ and  $F(H)$ is therefore also finite. However, in $\mathbf{\mathcal{C}_{n, \phi}}$, the coproduct $F(G)\oplus F(H)$ is, as we saw above, necessarily infinite (unless  both $G$ and $K$ are trivial). Therefore the coproduct cannot be the wedge sum. It would be interesting to find the graph operation which plays the role of the categorical coproduct in $\mathbf{\mathcal{C}_{n, \phi}}$.
    
    \subsection{Image}

    The categorical image of a morphism $f:G\to H$ in a category is a monomorphism $m:I\to H$ satisfying the following universal property:

      \begin{itemize}
    \item[(i)] There exists a morphism $e:G\to I$ such that $f=me$, and
    \item[(ii)] For any object $I'$ with a morphism $e':G\to I'$ and a monomorphism $m':I'\to Y$ such that $f=m'e'$, there exists a unique morphism $v:I\to I'$ such that $m=m'v$.

     \end{itemize}
In the category of groups, the categorical image of a morphism $f:G\to H$ is the standard image $\ima f = \{f(x) : x\in G\}$ together with the inclusion morphism $i: \ima f \to H$.

    \begin{corollary} The image of the morphism $F(f):F(G)\to F(H)$ in $\mathbf{\mathcal{C}_{n, \phi}}$ exists and is the monomorphism $F(i): F(\ima f)\to F(H)$.
        
    \end{corollary}
    One may intuitively define the image of a graph homomorphism $F(f):F(G) \to F(K)$ as the graph with vertex set $V(\ima F(f))=\{y \in F(K):y=F(f)(x) \text{ for some }x\in F(G)\}$ and edges given by the evolution map $\phi$. This agrees with the above equation since $V(\ima F(f))=\{y\in K^n:y_i=f(x_i) \text{ for some $x_i \in G$}\}=V(F(\ima f))$.

    \subsection{Pullbacks}
    In the category of groups, the pullback of two morphisms $f:X\to Z$ and $g:Y \to Z$ is given by the group $P=\{(x,y)\in G\times Y:f(x)=g(y)\}$ with the usual projections $\pi_1$ and $\pi_2$, as illustrated by the following diagram:
    \begin{center}
        \begin{tikzcd}
        P \arrow[r, "\pi_1"] \arrow[d, "\pi_2"'] & X \arrow[d, "f"] \\
        Y \arrow[r, "g"']                        & Z               
        \end{tikzcd}
    \end{center}
    By taking the image in $F$ of this diagram we get the pullback in $\mathbf{\mathcal{C}_{n, \phi}}$.

    \subsection{Short exact sequences}
    As an equivalence of categories, $F$ already preserves short exact sequences. One may see this explicitly by considering an arbitrary short exact sequence:
        \[
        \begin{tikzcd}
        \{e\} \arrow[r] & K \arrow[r, "f"] & G \arrow[r, "g"] & H \arrow[r] & \{e\}
        \end{tikzcd}
        \]
    and its image in $F$:
         \[
         \begin{tikzcd}
        F(\{e\}) \arrow[r] & F(K) \arrow[r, "F(f)"] & F(G) \arrow[r, "F(g)"] & F(H) \arrow[r] & F(\{e\})
        \end{tikzcd}
        \]
    From sections 4.1 and 4.5, we have that $F(f)$ is injective and $F(g)$ is surjective. Additionally, we have that  $\ima(F(f))=F(\ima f)=F(\ker g)=\ker F(g)$.\newline
    Since the functor preserves short exact sequences, we may now apply standard techniques of homological algebra (snake lemma,  five lemma, etc, whenever this makes sense) to automata, something that to our knowledge has not been done before. 
    \begin{example}
        For any $n\geq3$, we have the following short exact sequence of groups, where $D_n=\langle \alpha, \rho : \alpha^2=\rho^n=1, \alpha\rho=\rho^{-1}\alpha\rangle$ is the dihedral group of order $2n$:
        \begin{center}
            \begin{tikzcd}
            \{e\} \arrow[r] & \mathbb Z_n \arrow[r, "i"] & D_n \arrow[r, "f"] & \mathbb Z_2 \arrow[r] & \{e\}
            \end{tikzcd}
        \end{center}
        with $i(x)=\rho^x$, and $f(\sigma)=\begin{cases}
            0 \text{ if } \sigma= \rho^k \\
            1 \text{ if } \sigma = \alpha \rho^k
        \end{cases}$. We can see that $i$ is injective and $f$ is surjective. Adding the parity morphism $\delta:S_n \to \mathbb Z_2$ gives:
        \begin{center}
            \begin{tikzcd}
            & & & S_n \arrow[d, "\delta"]& \\
            \{e\} \arrow[r] & \mathbb Z_n \arrow[r, "i"] & D_n \arrow[r, "f"] & \mathbb Z_2 \arrow[r] & \{e\}
            \end{tikzcd}
        \end{center}
        Then we construct the pullback $P_n = \{(x,y)\in D_n\times S_n : f(x)=\delta(y)\}$:
        \begin{center}
            \begin{tikzcd}
            &                            & P_n \arrow[r, "\pi_2"] \arrow[d, "\pi_1"] & S_n \arrow[d, "\delta"] &   \\
            \{e\} \arrow[r] & \mathbb Z_n \arrow[r, "i"] & D_n \arrow[r, "f"]                        & \mathbb Z_2 \arrow[r]   & \{e\}
            \end{tikzcd}
        \end{center}
        We then take the kernel of the morphism $\pi_2$, which the subgroup $\ker \pi_2=\{(x,y)\in P_n : y=e\}= \langle \rho\rangle \times \{e\}$. From basic homological results, we get the following commutative diagram where the rows are exact and the morphism $\mu$ is an isomorphism:
        \begin{center}
            \begin{tikzcd}
            \{e\} \arrow[r] \arrow[d] & \langle \rho\rangle \times \{e\} \arrow[r] \arrow[d, "\mu"] & P_n \arrow[r, "\pi_2"] \arrow[d, "\pi_1"] & S_n \arrow[d, "\delta"] \arrow[r] & \{e\} \arrow[d]  \\
            \{e\} \arrow[r]           & \mathbb Z_n \arrow[r, "i"]                           & D_n \arrow[r, "f"]                        & \mathbb Z_2 \arrow[r]             & \{e\}          
            \end{tikzcd}
        \end{center}
        The bottom line is in fact a right-split exact sequence, since there exists an injective morphism (in fact, $n$ morphisms) $g:\mathbb Z_2 \to D_n$ sending $0\mapsto e$, and $1 \in \mathbb Z_2$ to any of the $n$ elements of order $2$ in $D_n$ (i.e. the reflections) such that $f \circ g=1_{\mathbb Z_2}$. This makes sense since $D_n \cong \mathbb Z_n \rtimes \mathbb Z_2$. So we have $n+2$ morphisms into $D_n$, one of which is surjective ($\pi_1$), the others being injective. Note that the kernel of the surjective morphism $\delta$ is $A_n$, giving rise to another right-split exact sequence  $\{e\}\to A_n\to S_n\to \mathbb Z_2\to \{e\}$.   Applying the functor $F$ to this diagram yields the commutative exact diagram in $\mathbf{\mathcal C_{n,\phi}}$:

        \begin{center}
            \begin{tikzcd}
            
            \{e^n\} \arrow[r] \arrow[d] & F(\langle \rho\rangle \times \{e\}) \arrow[r] \arrow[d, "F(\mu)"] & F(P_n) \arrow[r, "F(\pi_2)"] \arrow[d, "F(\pi_1)"] & F(S_n) \arrow[d, "F(\delta)"] \arrow[r] & \{e^n\} \arrow[d]  \\
            \{e^n\} \arrow[r]           & F(\mathbb Z_n) \arrow[r, "F(i)"]                           & F(D_n) \arrow[r, "F(f)"]                        & F(\mathbb Z_2) \arrow[r]             & \{e^n\}          
            \end{tikzcd}
        \end{center}
    \end{example}

    The structure of the graph $F(D_n)$ can now be studied via the structures of the graphs $F(P_n)$, $F(\mathbb Z_n)$ and $F(\mathbb Z_2)$.  As we will see in the next section, the fact that $F(i)$ and $F(g)$ are injective, and $F(\pi_1)$ and $F(f)$ are surjective, will allow us the infer information on the cyclic structure of $F(D_n)$.

\section{Cyclic structure}
We now turn to study the cyclic structure of the state transition graphs. That is, we look to use the results obtained in the previous section to infer cyclic information about an automaton based on the underlying group and its properties.

\begin{remark}
If $f:G_1\to G_2$ is a group homomorphism, and $C$ is a cycle in $F(G_1)$, then since $F(f)$ is a graph morphism between functional graphs, $F(f)$ maps $C$ onto a cycle in $F(G_2)$ whose length must divide the length of $C$ (for us a loop is a cycle). 
    
\end{remark}

    \begin{remark}
        If $f:G_1\to G_2$ is an injective group homomorphism, then $F(G_1)$ is isomorphic to an induced subgraph of $F(G_2)$. In particular, if $F(G_1)$ has $k$ cycles of length $l$, then $F(G_2)$ has at least as many. For example, using the first of the Sylow Theorems, we have that if $p$ is a prime power dividing the order of $G$, then all cycles of $F(\mathbb Z_{p})$ are also in $F(G)$, since $G$ has a subgroup of order $p$ which is isomorphic to $\mathbb Z_{p}$.
    \end{remark}
    
\begin{remark} If $f:G_1\to G_2$ is surjective, then any cycle $D$ in $F(G_2)$ must be the homomorphic image by $F(f)$ of a cycle $C$ in $F(G_1)$.  Indeed, since $F(f)$ is surjective, for a given vertex $y\in D$ there is a vertex $x$ in $F(G_1)$ such that $F(f)(x)=y$.  Now if $x_1$ is the evolution of $x$, then $F(f)(x_1)$ must be the evolution of $y$, and hence on the cycle $D$ since $F(G_2)$ is a functional graph. Continuing this process, we construct a sequence $\{x=x_0,x_1, \dots\}$ in $F(G_1)$ where $x_i$ evolves to $x_{i+1}$ for each $i$, and $F(f)(x_i)$ evolves to $F(f)(x_{i+1})$ on the cycle $D$.  This sequence cannot be infinite, so will eventually give rise to a cycle $C$ which by construction must satisfy $F(f)(C)\subseteq D$.  Now from above $F(f)(C)$ is a cycle, contained in $D$, hence must be equal to $D$, since all cycles in a functional graph are simple.

 \end{remark}

 Combining these facts we get the following proposition:
\begin{proposition} Given a short exact sequence of groups

     \[
        \begin{tikzcd}
        \{e\} \arrow[r] & K \arrow[r, "f"] & G \arrow[r, "g"] & H \arrow[r] & \{e\}
        \end{tikzcd}
        \]
    The lengths of the cycles in $F(K)$ and $F(H)$ are bounded by the maximal length of the cycles in $F(G)$.

    \end{proposition}

\begin{example}
    In the previous example, since we have surjective morphisms $f:D_n\to \mathbb Z_2$ and $\pi_1:P_n \to  D_n$, we get that the lengths of the cycles in $F(D_n)$ are bounded by those of $F(\mathbb Z_2)$ and $F(P_n)$, and similarly, the lengths of the cycles in $F(S_n)$ are bounded by those of $F(\mathbb Z_2)$ and $F(P_n)$.

    Also, from injectivity, we get that the cycles of $F(\mathbb Z_n)$ and $F(\mathbb Z_2)$ all lie in $F(D_n)$.  Similarly $F(P_n)$ will contain all the cycles lying in $F(\mathbb Z_n)$ and $F(A_n)$ (since, from the snake lemma, it is the kernel of $F(\pi_1)$.

    This allows us to transfer information between graphs.
    
\end{example}

We also have the following result on reversibility:

    \begin{proposition}
        Given an surjective homomorphism $f:G_1\to G_2$, if $F(G_1)$ is surjective (reversible), then so is $F(G_2)$.
    \end{proposition}
    \begin{proof}
        Take $y \in F(G_2)$, by proposition 4, $F(f)$ is surjective, hence $y = F(f)(x)$ for some $x \in F(G_1)$. Since $F(G_1)$ is surjective, $x=\phi(a)$ for some $a \in F(G_1)$. So we have that $y=F(f)(x)=F(f)(\phi(a))=\phi(F(f)(a))$ and so $F(G_2)$ is surjective.
    \end{proof}
    As previously mentioned, the idea of product of automata has been studied before \cite[Ex. 1.17]{cecchSMM}\cite{castillocatprod}\cite{dennunziodyn}. We now give a way to obtain the full cyclic structure of this automaton from that of the factor automata. We denote by $m_l(\Gamma)$ the number of cycles of length $l$ in the graph $\Gamma$, and by $S_{\Gamma}$ the set of cycle lengths in the graph $\Gamma$.
    \begin{definition}[\cite{hammackhandbook}]
        The disjoint union of two graphs $\Gamma_1$ and $\Gamma_2$, denoted by $\Gamma_1 \oplus \Gamma_2$ is the graph whose vertex set and edge sets are the disjoint unions of those of the component graphs. (Note that we assume that $V(\Gamma_1)\cap V(\Gamma_2) = \emptyset$.) For any $k\in \mathbb{N}$, we denote by $k\cdot\Gamma$ the disjoint union of $k$ copies of the graph $\Gamma$.
    \end{definition}

    \begin{lemma}[\cite{hammackhandbook}]
        If $\Gamma_1$, $\Gamma_2$, and $\Gamma_3$ are graphs, then the tensor product distributes over the disjoint union:
        \begin{align}
            \Gamma_1 \otimes (\Gamma_2 \oplus \Gamma_3) = (\Gamma_1 \otimes \Gamma_2) \oplus (\Gamma_1\otimes \Gamma_3)
        \end{align}
    \end{lemma}
    
    \begin{lemma}[\cite{hammackhandbook}]
        Let $C_k$ denote a graph consisting of a single cycle of length $k$. Then 
        \begin{align}
            C_p \otimes C_q \cong \gcd(p,q)\cdot C_{\lcm(p,q)}
        \end{align}
    \end{lemma}
    \begin{corollary}
        The possible cycles lengths $S_{\Gamma_1 \otimes \Gamma_2}$ in $\Gamma_1 \otimes \Gamma_2$ are:
        \begin{equation}
            S_{\Gamma_1 \otimes \Gamma_2}=\{\lcm(p,q):p \in S_{\Gamma_1} \text{ and } q \in S_{\Gamma_2}\}
        \end{equation}
    \end{corollary}
    
    \begin{theorem}
        The number $ m_l$ of cycles of length $l$ in $\Gamma_1 \otimes \Gamma_2$ is given by
        \begin{equation}
            m_l\big(\Gamma_1 \otimes \Gamma_2\big)=\sum_{\substack{p\in S_{\Gamma_1}, q \in S_{\Gamma_2}\\ \lcm(p,q)=l} }\gcd(p,q) m_p(\Gamma_1)m_q(\Gamma_2)
        \end{equation}
    \end{theorem}
    \begin{proof}
        By Lemma 2, each pair of cycles of lengths $p,q$ with $\lcm(p,q)=l$ will contribute $\gcd(p,q)$ cycles of length $l$. By Lemma 1, the tensor product distributes over the disjoint union, allowing us to apply this component-wise. Summing over all such $p,q$ cycles lengths in $\Gamma_1, \Gamma_2$ respectively yields the result. Note that transient behavior in the factor graphs neither contributes to nor strips away cyclic behavior in the product, hence Equation 3. does not account for it.
    \end{proof}
    Together with the Fundamental Theorem of Finite Abelian Groups, Theorems 1 and 2 allow us to compute the cycle structure of the state transition graphs of an automaton over any finite abelian group alphabet using that of the component cyclic groups. That is, for any $G \cong \prod_{i \in I} \mathbb{Z}_i$, knowing $\{F_{n,\phi}(\mathbb{Z}_i)\}_{i \in I}$ one can compute both $F_{n, \phi}(G)$ and its cycle structure.
    
    \begin{example}
        We use Theorems 1 and 2 to calculate the multiplicities of cycles in the graph $F_{17, \phi_{90}}(K_4)$ (where $K_4$ is the Klein 4-group) using the fact that $K_4 \cong \mathbb{Z}_2\times \mathbb{Z}_2$. We reproduce the $F(\mathbb{Z}_2)$ cycle information from Table 2. of \cite{wolfram1984} in the first and second column:
        
        \begin{center}
    \begin{tabular}{ |c|c|c| } 
    
      Cycle length & Multiplicity in $F(\mathbb{Z}_2)$ & Multiplicity in $F(K_4)$ \\ 
      \hline
      $1$ & $1$ & $1$\\  
      $5$ & $51$ & $13,107$ \\  
      $15$ & $4,352$ & $286,326,784$   
    \end{tabular}
    \end{center}
    \end{example}
    S. Wolfram shows in \cite{wolfram1984} that over $\mathbb Z_p$, the lengths of all cycles divide the length of the maximal length cycle. For an arbitrary group however, this is not true in general. For example, see Figure \ref{fig:S3_N3_R90}, where in the state transition graph of a Rule 90 automaton over $S_3$, there are 6-cycles, but 6 does not divide 8, which is the maximal cycle length.
    
    \begin{figure}[htbp]
        \centering
        \includegraphics[width=1\textwidth]{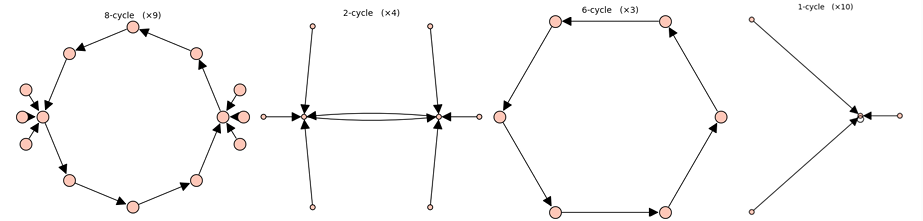} 
        \caption{The state transition graph $F_{3,\phi_{90}}(S_3)$ shows how divisibility of cycle lengths can fail for automata over non-abelian alphabets.}
        \label{fig:S3_N3_R90}
    \end{figure}
    
    \begin{theorem}
        When $G$ is a finite abelian group, all cycle lengths in the graphs of the factor cyclic groups divide the maximal cycle length in $F(G)$ Furthermore, all cycles lengths in $F(G)$ also divide the maximal cycle length.
    \end{theorem}
    \begin{proof}
        By the Fundamental Theorem of Finite Abelian Groups, $G \cong \mathbb Z_{{p_1}^{e_1}}\times\cdots \times \mathbb Z_{{p_k}^{e_k}}$, hence by Theorem 1, $F(G)\cong F(Z_{{p_1}^{e_1}})\otimes \cdots \otimes F(Z_{{p_k}^{e_k}})$. It follows that the maximal cycle length in $F(G)$, denoted as $\pi_G$, is the least common multiple of the maximal cycle lengths in the factor graphs, hence they all divide it. Now in the $F(\mathbb Z_{{p_i}^{e_i}})$, methods from \cite{wolfram1984} can be used to show all cycle lengths divide the maximal length. It follows by transitivity of division that they also divide $\pi_G$. Additionally, since the least common multiple is monotone with respect to divisibility, and since $a_i \mid a_{max}, b_j\mid b_{max}, \cdots , k_l \mid k_{max}$, then $\lcm(a_i,b_j,\dots,k_l)\mid \lcm(a_{max},b_{max},\dots, k_{max})=\pi_G$ where $a_i$ is a cycle length in the first factor graph, and $a_{max}$ its maximal cycle length, etc. Since every cycle length in $F(G)$ is of this form, it must divide $\pi_G$.
    \end{proof}
    Another result in \cite{wolfram1984} which does not generalize to non-abelian cases is the uniformity of indegrees mentioned in Section 2, that is when $G$ is abelian, the size of evolution preimages are either $0$ or exactly $\lvert\ker\phi\rvert$, . However, in non-abelian cases this does not hold, neither do preimages divide the kernel size. Take for example the state transition graph in Figure \ref{Q8_N3_R90}. Several nodes have indegree $6$ whereas the identity node, which must be a fixed point since the quiescence condition is respected, has indegree $2$.
    \begin{figure}[htbp]
        \centering
        \includegraphics[width=0.8\textwidth]{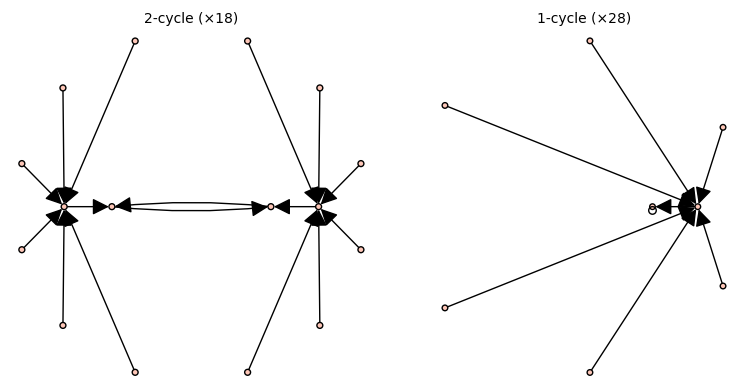} 
        \caption{The state transition graph $F_{3,\phi_{90}}(Q_8)$ shows how divisibility of indegrees can fail for automata over non-abelian groups.}
        \label{Q8_N3_R90}
    \end{figure}
    
    \section{Discussion and Conclusion}

In this paper, we introduced a functorial framework to bridge abstract algebra and the dynamical properties of cellular automata. By defining the category $\mathcal{C}_{n,\phi}$ of state transition graphs of group-alphabet cellular automata (GACA) and constructing a faithful functor $F_{n,\phi}: \mathbf{Grp} \to \text{pfG}$, we established an equivalence of categories between $\mathbf{Grp}$ and $\mathcal{C}_{n,\phi}$. 

This equivalence allows for a direct dictionary between algebraic operations on group alphabets and graph-theoretical transformations on state transition spaces:
\begin{itemize}
    \item Categorical products in $\mathbf{Grp}$ correspond precisely to graph tensor products in $\mathcal{C}_{n,\phi}$, yielding exact formulas for cycle lengths and multiplicities in product automata $F(G_1 \times G_2)$.
    \item Standard algebraic constructions---\allowbreak including kernels, cokernels, pullbacks, and short exact sequences---\allowbreak translate directly into exact graph inclusions, quotient graphs, and commutative diagrams of state transition graphs.
\end{itemize}

Several open questions and avenues for future research arise naturally from this work:
\begin{enumerate}
    \item \textbf{Coproduct Structures in $\text{pfG}$:} While the product in $\mathcal{C}_{n,\phi}$ corresponds to the graph tensor product, the categorical coproduct in $\mathcal{C}_{n,\phi}$ is given by the free product $F(\oplus_{i \in I} G_i)$, which yields an infinite state transition graph. Identifying an explicit graph-theoretic operator on functional digraphs that mimics this free product structure remains an open question.
    \item \textbf{Generalization to Higher Dimensions and Different Lattices:} The functorial construction presented here focuses on one-dimensional lattices $\mathbb{Z}_n$ with periodic boundary conditions. Extending this framework to $d$-dimensional grid topologies $\mathbb{Z}_n^d$ or general cayley graphs as lattices could reveal further connections between spatial symmetries and dynamical categories.
    \item \textbf{Categorical Invariants and Homological Applications:} Given that $F$ preserves short exact sequences, further investigation into homological invariants (e.g., homology groups of associated chain complexes on state space graphs) could yield new computational tools for analyzing fixed points, Gardens-of-Eden, and transient dynamics in complex non-abelian automata.
\end{enumerate}

\section*{Acknowledgments}
This work was funded by a CRM–ISM Undergraduate Research Scholarship.

\printbibliography

@article{wolfram1984,
  title={Algebraic properties of cellular automata},
  author={Martin, O. and Odlyzko, A. M. and Wolfram, S.},
  journal={Communications in Mathematical Physics},
  volume={93},
  pages={219--258},
  year={1984},
}

@book{hammackhandbook,
  title={Handbook of product graphs},
  author={Hammack, R. and Imrich, W. and Klav{\v{z}}ar, S.},
  year={2011},
  edition={2nd},
  publisher={CRC Press}
}

@misc{castillocatprod,
      title={Categorical products of cellular automata}, 
      author={A. Castillo-Ramirez and A. Vazquez-Aceves and A. Zaldivar-Corichi},
      year={2025},
      eprint={2503.21567},
      archivePrefix={arXiv},

}

@book{cecchSMM,
  title     = {Cellular Automata and Groups},
  author    = {Ceccherini-Silberstein, T. and Coornaert, M.},
  series    = {Springer Monographs in Mathematics},
  year      = {2024},
  publisher = {Springer},
}

@book{roylealgebraicgraph,
   title =     {Algebraic Graph Theory},
   author =    {C. Godsil, G. Royle},
   publisher = {Springer},
   year =      {2001},
   series =    {Graduate Texts in Mathematics},
}

@article{dennunziodyn,
  author       = {A. Dennunzio and
                  E. Formenti and
                  L. Margara},
  title        = {On the Dynamical Behavior of Cellular Automata on Finite Groups},
  journal      = {IEEE Access},
  volume       = {12},
  pages        = {122061--122077},
  year         = {2024},
  
}

@article{bardzellevohomomo,
  title={The evolution homomorphism and permutation actions on group generated cellular automata},
  author={Miller, N. R. and Bardzell, M. J.},
  journal={Complex Systems},
  volume={15},
  number={2},
  pages={121--136},
  year={2004}
}

\end{document}